\documentclass[11pt]{amsart}
\usepackage{amsmath,amssymb,amsfonts,enumerate,tikz,url,booktabs,centernot}
\usepackage[group-minimum-digits=4,group-separator=\text{,}]{siunitx}

\newcommand{\Q}{\mathbb{Q}}
\newcommand{\Z}{\mathbb{Z}}

\newcommand{\F}{\mathbb{F}}
\newcommand{\Fp}{\F_p}
\DeclareMathOperator{\M}{\mathsf{M}}

\newcommand{\SSS}{\mathcal{S}}

\newcommand{\USp}{\mathrm{USp}}
\newcommand{\kron}[2]{\left(\frac{#1}{#2}\right)}
\newcommand{\inkron}[2]{(#1/#2)}
\newcommand{\pnum}[1]{\num{#1}\phantom{.0}}

\newtheorem{theorem}{Theorem}[section]
\newtheorem{lemma}[theorem]{Lemma}

\theoremstyle{definition}

\newtheorem{example}[theorem]{Example}
\newtheorem{remark}[theorem]{Remark}

\title[Computing Hasse--Witt matrices]
 {Computing Hasse--Witt matrices of hyperelliptic curves in average polynomial time, II}
 
\author{David Harvey and Andrew V. Sutherland}

\begin{document}

\begin{abstract}
We present an algorithm that computes the Hasse--Witt matrix of given hyperelliptic curve over $\Q$ at all primes of good reduction up to a given bound $N$.
It is simpler and faster than the previous algorithm developed by the authors.
\end{abstract}
 
 \maketitle

\section{Introduction}
\label{sec:intro}

Let $C/\Q$ be a (smooth projective) hyperelliptic curve of genus $g$ defined by an affine equation of the form $y^2=f(x)$,
with $f=\sum_i f_i x^i \in\Z[x]$ squarefree.
Primes $p$ for which the reduced equation $y^2= f(x)\bmod p$ defines a hyperelliptic curve $C_p/\Fp$ of genus $g$ are \emph{primes of good reduction} (for the equation $y^2=f(x)$).
For each such prime $p$, the \emph{Hasse--Witt matrix} (or \emph{Cartier--Manin matrix}) of $C_p$ is the $g\times g$ matrix $W_p=[w_{ij}]$ over $\Z/p\Z$ with entries
\[
w_{ij} = f^{(p-1)/2}_{pi-j}\bmod p\qquad (1\le i,j\le g),
\]
where $f^n_k$ denotes the coefficient of $x^k$ in $f(x)^n$; see \cite{Manin:HasseWittMatrix,Yui:HasseWittMatrix} for details.
The matrix~$W_p$ depends on the equation $y^2=f(x)\bmod p$ for the curve $C_p$, but its conjugacy class, and in particular, its characteristic polynomial, is an invariant of the function field of $C_p$.

The Hasse--Witt matrix $W_p$ is closely related to the \emph{zeta function}
\begin{equation}\label{eq:zeta}
Z_p(T):=\exp\left(\sum_{k = 1}^\infty\frac{\#C_p(\F_{p^k})}{k}T^k\right)=\frac{L_p(T)}{(1-T)(1-pT)}.
\end{equation}
Indeed, the numerator $L_p\in\Z[T]$ satisfies
\[
L_p(T)\equiv \det(I-TW_p)\pmod p,
\]
and we also have
\[
\chi_p(T)\equiv (-1)^gT^g\det(W_p-TI)\pmod p,
\]
where $\chi_p(T)$ denotes the characteristic polynomial of the Frobenius endomorphism of the Jacobian of $C_p$.
In particular, the trace of $W_p$ is equal to the trace of Frobenius modulo $p$, and for $p> 16g^2$ the Riemann Hypothesis for curves implies that this relationship uniquely determines the trace of Frobenius.

In this paper we present an algorithm that takes as input the polynomial $f(x)$ and an integer $N>1$, and simultaneously computes $W_p$ for all primes $p\le N$ of good reduction.
Our main result is the complexity bound given in Theorem~\ref{thm:main} below; the details of the algorithm are given in \S \ref{sec:algs}.

All time complexity bounds refer to bit complexity. We denote by $\M(s)$ the time needed to multiply $s$-bit integers; we may take $\M(s) = O(s \, (\log s)^{1+o(1)})$ (see \cite{SS:IntegerMultiplication, Furer2009}, or \cite{HLvdH-zmult} for recent improvements).
We assume throughout that $\M(s)/(s \log s)$ is increasing, and that the space complexity of $s$-bit integer multiplication is $O(s)$.
\begin{theorem}\label{thm:main}
Assume that $g \log g = O(\log N)$ and that $\log\max_i |f_i| = O(\log N)$.
The algorithm \textsc{ComputeHasseWittMatrices} computes $W_p$ for all primes $p \leq N$ of good reduction in $O(g^3 \M(N\log N)\log N)$ time and $O(g^2 N)$ space.
\end{theorem}



The average running time of \textsc{ComputeHasseWittMatrices} per prime $p\le N$ is
\[
O(g^3(\log p)^{4+o(1)}),
\]
which is polynomial in $g \log p$, the bit-size of the equation defining $C_p$.
While it is known that one can compute the characteristic polynomial of $W_p$ for any particular prime $p$ in time polynomial in $\log p$ (the Schoof--Pila algorithm \cite{Schoof:Polytime,Pila:Polytime}), or polynomial in~$g$ (Kedlaya's algorithm \cite{kedlaya:algorithm}, for example), we are aware of no algorithm that can compute even the trace of $W_p$ in time that is polynomial in $g\log p$.

The algorithm presented here improves the previous algorithm given by the authors in~\cite{HS:HyperellipticHasseWitt},
which was in turn based on the approach introduced in \cite{Harvey:HyperellipticPolytime}.
The new algorithm is easier to describe and implement, and it is significantly faster.
Asymptotically we gain a factor of~$g^2$ in the running time:
one factor of $g$ arises from genuine algorithmic improvements in the present paper, and another factor of $g$ follows from unrelated recent work on the theoretical complexity of integer matrix multiplication \cite{HvdH:zmatmult}.
The new algorithm also uses less memory by a significant constant factor.
Tables~\ref{table:hw1vshw2g2} and ~\ref{table:hw1vshw2g3} in \S\ref{sec:perf} give performance comparisons in genus 2 and 3, where the the new algorithm is already up to 8 times faster.
Compared to previous methods for solving this problem (i.e., prior to \cite{Harvey:HyperellipticPolytime}), the new algorithm is dramatically faster, with more than a 300-fold speed advantage for $N=2^{30}$; see Table~\ref{table:sjhfcmp}.
Performance results for hyperelliptic curves of genus $g\le 10$ can be found in Table~\ref{table:highergenus}.

In addition to the average polynomial-time algorithm, we give an $O(g^2p \M(\log p))$ algorithm to compute $W_p$ for a single prime $p \geq g$.
While the dependence on~$p$ is not asymptotically competitive with existing algorithms, it is easy to implement, and uses very little memory.
The small constant factors in its complexity make it a good choice for small to medium values of $p$; see Table~\ref{table:modp}.

We also introduce techniques that may be of interest beyond the scope of our algorithmic applications.
In particular, we show that the matrix $W_p$ for a given curve may be derived from knowledge of just the \emph{first} row of the matrices $W_p$ corresponding to $g$ isomorphic curves.
Algorithmically, this has the advantage that we never need to go beyond the coefficient of $x^{p-1}$ in the expansion of $f(x)^{((p-1)/2)}$ in order to compute $W_p$.

\section{Recurrence relations}
\label{sec:RR}

As above, let $C/\Q$ be a hyperelliptic curve of genus $g$.
We may assume without loss of generality that $C$ is defined by an equation of the form
\[
y^2=f(x)=\sum_{i=c}^d f_ix^i\qquad(f_i\in\Z)
\]
with $c\in\{0,1\}$, $d\in \{2g+1,2g+2\}$ and $f_cf_d\ne 0$ (we have $c\le 1$ because $f$ is squarefree).
It is convenient to normalize by taking $h(x) := f(x)/x^c$ and $r := d - c$.
Then
\[
h(x)=\sum_{i=0}^r h_i x^i
\]
where $h_i = f_{i+c}$ for $i = 0, \ldots, r$, and $h_0 h_r \neq 0$.

We now derive a recurrence for the coefficients $h^n_k$ of $h^n$, following the strategy of Bostan--Gaudry--Schost \cite[\S8]{BGS-recurrences}.
The identities
\[
h^{n+1} = h \cdot h^n\qquad\text{and}\qquad(h^{n+1})' = (n+1)h'\cdot h^n
\]
yield the relations
\[
h_k^{n+1}=\sum_{j=0}^r h_j h_{k-j}^n\qquad\text{and}\qquad k h_k^{n+1}=(n+1)\sum_{j=0}^r j h_j h_{k-j}^n.
\]
Subtracting $k$ times the first relation from the second and solving for $h^n_k$ yields the recurrence
\begin{equation}\label{eq:up}
h^n_k = \frac{1}{k h_0} \sum_{j=1}^r ((n+1) j - k) h_j h^n_{k-j},
\end{equation}
which expresses $h^n_k$ in terms of the previous $r$ coefficients $h^n_{k-1}, \ldots, h^n_{k-r}$, for all $k > 0$.

The recurrence may be written in matrix form as follows. Define the integer row vector
\[
v_k^n := [h^n_{k-r+1},\ldots,h^n_k] \in \Z^r.
\]
Then
\[
v^n_k = \frac1{k h_0} v^n_{k-1}M^n_k,
\]
where
\begin{equation}
\label{eq:matrix}
M^n_k := \begin{bmatrix}
0 &   \cdots &  0 & ((n+1)r-k)h_r\\
k h_0 &  \cdots  & 0 & ((n+1)(r-1)-k)h_{r-1}\\
\vdots & \ddots & \vdots &\vdots\\
0 &  \cdots  & k h_0 & (n+1-k)h_1
\end{bmatrix}.
\end{equation}
Iterating the recurrence, we obtain an explicit formula for the $m$th term, for any $m \geq 0$:
\[
v^n_m = \frac{1}{(h_0)^m m!} v^n_0 M^n_1 \cdots M^n_m.
\]
Since the initial vector is simply $v^n_0 = [0, \ldots, 0, (h_0)^n]$, we may rewrite this as
\begin{equation}
\label{eq:vnm}
v^n_m = \frac{1}{(h_0)^{m-n} m!} V_0 M^n_1 \cdots M^n_m,
\end{equation}
where $$V_0 := [0, \ldots, 0, 1] \in \Z^r.$$

Everything discussed so far holds over $\Z$.
Now consider a prime $p$ of good reduction that does not divide~$h_0$, and let $W_p$ be the Hasse--Witt matrix of $y^2 = f(x)\bmod p$, as in \S\ref{sec:intro}.
Write $W_p^i$ for the $i$th row of $W_p$.
Specializing the above discussion to $n := (p-1)/2$, the entries of $W_p^i$ are given by $w_{ij} = f^n_{pi-j} = h^n_{pi-j-cn} \bmod p$ for $j=1, \ldots, g$.
These are the last~$g$ entries, in reversed order, of the vector $v^n_{pi-cn-1} \bmod p$, which by \eqref{eq:vnm} is equal to
\begin{equation}
\label{eq:vnm-spec}
\frac{1}{(h_0)^{(pi-cn-1)-n}(pi-cn-1)!} V_0 M^n_1 \cdots M^n_{pi-cn-1} \pmod p.
\end{equation}

\begin{remark}
Bostan, Gaudry and Schost \cite{BGS-recurrences} used a formula essentially equivalent to \eqref{eq:vnm-spec} to compute $W_p$, for a single prime $p$.
Their innovation was to evaluate the matrix product using an improvement of the Chudnovsky--Chudnovsky method, leading to an overall complexity bound of $g^{O(1)} p^{1/2+o(1)}$ for computing $W_p$.
\end{remark}

\begin{remark}
The power of $p$ dividing the denominator $(pi-cn-1)!$ in \eqref{eq:vnm-spec} is at least $p^{i-1}$.
In particular, if $i \geq 2$, the denominator is divisible by $p$.
Thus, to compute the second and subsequent rows of $W_p$, the algorithm in \cite{BGS-recurrences} artificially lifts the input polynomial $f \in \Fp[x]$ to $(\Z/p^\lambda \Z)[x]$ for a suitable $\lambda \geq 2$, and then works modulo $p^\lambda$ throughout the computation, reducing modulo $p$ at the end.
In this paper we compute $W_p$ by computing the first row of~$g$ conjugate Hasse--Witt matrices (as explained in \S\ref{sec:Wpa}), so we can work modulo $p$ everywhere.
\end{remark}

We now focus on $W_p^1$, the first row of $W_p$.
Taking $i=1$ in \eqref{eq:vnm-spec} and putting $e := 2 - c$, we find that $W_p^1$ is given by the last $g$ entries of
\begin{equation}
\label{eq:ven}
v^n_{en} \equiv \frac{1}{(h_0)^{n(e-1)}(en)!} V_0 M^n_1 \cdots M^n_{en} \pmod p.
\end{equation}
To compute $W_p^1$, it suffices to evaluate the vector-matrix product $V_0 M^n_1 \cdots M^n_{en}$ modulo~$p$, since, having assumed $p\centernot | h_0$, the denominator $(h_0)^{n(e-1)}(en)!$ is not divisible by $p$ (note that $e\in\{1,2\}$, so $en$ is at most $p-1$).

\begin{remark}
\label{rem:root}
The quantity $(h_0)^n=(h_0)^{(p-1)/2}$ is just the Legendre symbol $\inkron{h_0}{p}=\pm 1$, and
the denominator $\inkron{h_0}{p}^{e-1}(en)!$ is always a fourth root of unity modulo~$p$.
Indeed, if $e = 2$ then $(en)! = (p-1)! \equiv -1 \pmod p$, by Wilson's theorem; if $e=1$ then it is well known that $(en)! = ((p-1)/2)!$ is a fourth root of unity modulo $p$.
However, we know of no easy way to determine \emph{which} fourth root of unity occurs.
For example, if $p>3$ and $p\equiv 3\pmod 4$, then $((p-1)/2)!\equiv \pm 1\pmod p$, where the sign depends on the class number of $\Q(\sqrt{-p})$ modulo $4$; see \cite{Mordell:SqrtWilson}.
\end{remark}

To evaluate \eqref{eq:ven} simultaneously for many primes, a crucial observation is that the matrix $M_k^n$ becomes ``independent of $n$'' after reduction modulo $p = 2n+1$.
More precisely, we have $2(n+1) = 1 \pmod p$, so $2 M_k^n \equiv M_k \pmod p$ where
\begin{equation}
\label{eq:matrix2}
M_k := \begin{bmatrix}
0 &   \cdots &  0 & (r-2k)h_r\\
2 k h_0 &  \cdots  & 0 & (r-1-2k)h_{r-1}\\
\vdots & \ddots & \vdots &\vdots\\
0 &  \cdots  & 2 k h_0 & (1-2k)h_1
\end{bmatrix}.
\end{equation}
Note that $M_k$ is defined over $\Z$, and, unlike $M^n_k$, it is independent of $p$.
Multiplying \eqref{eq:ven} by $2^{en}$ yields the following lemma.
\begin{lemma}
\label{lem:Wp1}
The first row $W_p^1$ of $W_p$ consists of the last $g$ entries of the vector
\begin{equation}
\label{eq:main}
\frac{\kron{2}{p}^e}{\kron{h_0}{p}^{e-1}(en)!} V_0 M_1 \cdots M_{en} \pmod p.
\end{equation}
\end{lemma}

In order to unify the indexing for the cases $e=1$ and $e=2$ we now define
\[
M'_k := \begin{cases} M_k & e=1,\\ M_{2k-1} M_{2k} & e=2. \end{cases}
\]
Then \eqref{eq:main} becomes
\begin{equation}
\label{eq:main-primed}
\frac{\kron{2}{p}^e}{\kron{h_0}{p}^{e-1}(en)!} V_0 M'_1 \cdots M'_n \pmod p.
\end{equation}
In the next section we will recall how to use the \emph{accumulating remainder tree} algorithm to evaluate the product $V_0 M'_1 \cdots M'_n \pmod p$ for many primes $p$ simultaneously.

\begin{remark}
The key difference between this approach and that of~\cite{HS:HyperellipticHasseWitt} is that here we express $v^n_k$ in terms of $v^n_{k-1}$, instead of expressing $v^k_{2k}$ in terms of $v^{k-1}_{2k-2}$.
In the terminology of~\cite{Harvey:HyperellipticPolytime}, we have replaced ``reduction towards zero'' by ``horizontal reduction''.
In both cases, the aim is to obtain recurrences whose defining matrices are independent of~$p$, so that the machinery of accumulating remainder trees can be applied.
In the original ``reduction towards zero'' method of~\cite{HS:HyperellipticHasseWitt}, the desired independence follows from the choice of indexing, i.e., because the superscript and subscript in $v^k_{2k}$ do not depend on $p$.
In the new method, the superscript in $v^n_k$ does depend on $p$, but \emph{after reduction modulo $p$} the recurrence matrices turn out to be independent of $p$ anyway.
The new method may also be viewed as a special case of (and was inspired by) the ``generic prime'' device introduced in~\cite{Harvey:ArithmeticSchemesPolytime}; see for example the proof of \cite[Prop.~4.6]{Harvey:ArithmeticSchemesPolytime}.

Computationally, the new approach has two main advantages.
First, the matrices $M_k$ are simpler, sparser, and have smaller coefficients than those in \cite{HS:HyperellipticHasseWitt}.
In fact, the formula for~$M_k$ in \cite{HS:HyperellipticHasseWitt} was so complicated that even for genus $2$ we did not write it out explicitly in the paper.
Second, as pointed out earlier, none of the denominators $k h_0$ used to compute the first row of $W_p$ are divisible by $p$, so we can work modulo $p$ throughout; in \cite{HS:HyperellipticHasseWitt} we needed to work modulo a large power of $p$ (at least $p^g$) in order to handle powers of~$p$ appearing in the denominators.
\end{remark}

\section{Accumulating remainder trees}
\label{sec:art}

In this section we recall how the accumulating remainder tree algorithm works and sharpen some of the complexity bounds given in \cite{HS:HyperellipticHasseWitt}.

We follow the notation and framework of \cite[\S4]{HS:HyperellipticHasseWitt}.
Let $b \geq 2$ and let $m_1, \ldots, m_{b-1}$ be a sequence of positive integer moduli.
Let $A_0, \ldots, A_{b-2}$ be a sequence of $r \times r$ integer matrices, and let $V$ be an $r$-dimensional integer row vector.
We wish to compute the sequence of reduced row vectors $C_1, \ldots, C_{b-1}$ defined by
\[
  C_n := V A_0 \cdots A_{n-1} \bmod m_n.
\]
For convenience, we define $m_0 := 1$, so $C_0$ is the zero vector, and we let $A_{b-1}$ be the identity matrix.
(To apply this to the recurrences in \S\ref{sec:RR}, we let $V=V_0$ and $A_k = M'_{k+1}$; note that the index $k$ is shifted by one place.)

In \cite[\S4]{HS:HyperellipticHasseWitt} we gave an algorithm \textsc{RemainderTree} that efficiently computes the vectors $C_1, \ldots, C_{b-1}$ simultaneously.
For the reader's convenience we now recall some details of this algorithm.
As in \cite{HS:HyperellipticHasseWitt}, for simplicity we assume that $b = 2^\ell$ is a power of $2$, though this is not strictly necessary.

We work with complete binary trees of depth $\ell$ with nodes indexed by pairs $(i,j)$ with $0\le i\le \ell$ and $0\le j < 2^i$.
For each node we define the intermediate quantities
\begin{align}
m_{i,j} &:= m_{j 2^{\ell-i}} m_{j 2^{\ell-i} + 1} \cdots m_{(j+1)2^{\ell-i}-1}, \notag \\
A_{i,j} &:= A_{j 2^{\ell-i}} A_{j 2^{\ell-i} + 1} \cdots A_{(j+1)2^{\ell-i}-1}, \label{eq:definetrees} \\
C_{i,j} &:= V A_{i,0} \cdots A_{i,j-1}  \bmod m_{i,j}. \notag
\end{align}
The values $m_{i,j}$ and $A_{i,j}$ may be viewed as nodes in a \emph{product tree}, in which each node is the product of its children,
with leaves $m_j=m_{\ell,j}$ and $A_j=A_{\ell,j}$, for $0 \leq j < b$.
Each vector $C_{i,j}$ is the product of $V$ and all the matrices $A_{i,j'}$ that are nodes on the same level and to the left of $A_{i,j}$, reduced modulo $m_{i,j}$.
The following algorithm, copied verbatim from \cite{HS:HyperellipticHasseWitt}, computes the target values $C_j = C_{\ell,j}$.

\bigskip

\noindent
\textbf{Algorithm} \textsc{RemainderTree}
\vspace{2pt}

\noindent
Given $V, A_0,\ldots,A_{b-1}$, $m_0,\ldots,m_{b-1}$, with $b=2^\ell$, compute $m_{i,j}, A_{i,j}$, $C_{i,j}$ as follows:
\smallskip

\begin{enumerate}[1.]
\setlength{\itemsep}{2pt}
\item Set $m_{\ell,j}=m_j$ and $A_{\ell,j}=A_j$, for $0\le j < b$.
\item For $i$ from $\ell-1$ down to 1:\\
\phantom{For} For $0\le j < 2^i$, set $m_{i,j}=m_{i+1,2j}m_{i+1,2j+1}$ and $A_{i,j}=A_{i+1,2j}A_{i+1,2j+1}$.
\item Set $C_{0,0}=V \bmod m_{0,0}$ and then for $i$ from 1 to $\ell$:\\
\phantom{Set }For $0 \leq j < 2^i$ set $C_{i,j} =
\begin{cases}
C_{i-1,\lfloor j/2\rfloor}\bmod m_{i,j}\qquad&\text{if $j$ is even,}\\ 
C_{i-1,\lfloor j/2\rfloor}A_{i,j-1}\bmod m_{i,j}&\text{if $j$ is odd.}\\ 
\end{cases}$
\end{enumerate}
\bigskip

A complexity analysis of this algorithm was given in \cite[Theorem~4.1]{HS:HyperellipticHasseWitt}, closely following the argument of \cite[Prop.\ 4]{Harvey:HyperellipticPolytime}.
The analysis assumed that classical matrix multiplication was used to compute the products $A_{i+1,2j}A_{i+1,2j+1}$ in step 2.
More precisely, defining $\M_r(s)$ to be the cost (bit complexity) of multiplying $r \times r$ matrices with $s$-bit integer entries, it was assumed that $\M_r(s) = O(r^3 \M(s))$.
Of course, this can be improved to $\M_r(s) = O(r^\omega \M(s))$, where $\omega \leq 3$ is any feasible exponent for matrix multiplication \cite[Ch.~12]{vzGG-compalg}.
For example, we can take $\omega = \log 7 / \log 2 \approx 2.807$ using Strassen's algorithm.

However, it was pointed out in \cite{HS:HyperellipticHasseWitt} and \cite{Harvey:HyperellipticPolytime} that one can do even better in practice, by reusing Fourier transforms of the matrix entries; heuristically this reduces the complexity to only $O(r^2 \M(s))$ when $s$ is large compared to~$r$.
Recently, a rigorous statement along these lines was established by van der Hoeven and the first author \cite{HvdH:zmatmult}.
The following slightly weaker claim is enough for our purposes.
\begin{lemma}\label{lem:zmatmult}
We have
 \[ \M_r(s) = O(r^2 \M(s) + r^\omega s (\log \log s)^2), \]
uniformly for $r$ and $s$ in the region $r = O(s)$.
\end{lemma}
\begin{proof}
According to \cite[Prop.~4]{HvdH:zmatmult} (which depends on our running hypothesis that $\M(s)/(s \log s)$ is increasing), in this region we have
 \[ \M_r(s) = O(r^2 \M(s) + r^\omega (s / \log s) \M(\log s)). \]
The desired bound follows, since certainly $\M(n) = O(n \log^2 n)$.
\end{proof}

Using this lemma, we can prove the following strengthening of \cite[Theorem 4.1]{HS:HyperellipticHasseWitt}.
\begin{theorem}
\label{thm:RT}
Let $B$ be an upper bound on the bit-size of $\prod_{j=0}^{b-1} m_j$, let $B'$ be an upper bound on the bit-size of any entry of $V$, and let $H$ be an upper bound on the bit-size of any $m_0, \ldots, m_{b-1}$ and any entry of $A_0, \ldots, A_{b-1}$.
Assume that $\log r = O(H)$ and that $r = O(\log b)$.
The running time of the \textsc{RemainderTree} algorithm is
\[
   O(r^2 \M(B + bH) \log b + r \M(B')),
\]
and its space complexity is
\[
   O(r^2 (B + bH) \log b + rB').
\]
\end{theorem}

This statement differs from \cite[Theorem 4.1]{HS:HyperellipticHasseWitt} in two ways: we have imposed the additional requirement that $r = O(\log b)$, and the $r^3 \M(B + bH) \log b$ term in the time complexity is improved to $r^2 \M(B + bH) \log b$ (we have also changed $h$ to $H$ to avoid a collision of notation).
\begin{proof}
Let us first estimate the complexity of computing the $A_{i,j}$ tree.
The entries of any product $A_{j_1} \cdots A_{j_2-1}$ have bit-size $O((j_2 - j_1)(H + \log r)) = O((j_2 - j_1)H)$.\footnote{There are some missing parentheses in the corresponding estimate in the proof of \cite[Theorem 4.1]{HS:HyperellipticHasseWitt}.}
Thus at level $\ell-i$ of the tree, each matrix product has cost
\[
  O(\M_r(2^i H)) = O(r^2 \M(2^i H) + r^\omega (2^i H) (\log \log 2^i H)^2),
\]
by Lemma~\ref{lem:zmatmult}.
There are $2^{\ell-i}$ such products at this level, whose total cost is
\[
  O(r^2 \M(bH) + r^\omega (bH) (\log \log bH)^2).
\]
Since we assumed that $r = O(\log b) = O(\log bH)$, this is bounded by
\[
  O(r^2 \M(bH) + r^2 (bh) (\log bH)^{\omega-2} (\log \log bH)^2).
\]
But we may take $\omega<3$, and then certainly $(bH) (\log bH)^{\omega-2} (\log \log bH)^2 = O(\M(bH))$, again by the assumption that $\M(s)/(s \log s)$ is increasing.
The first term dominates, and we are left with the bound $O(r^2 \M(bH))$ for this level of the tree, and hence
\[
  O(r^2 \M(bH) \log b)
\]
for the whole tree.

The rest of the argument is exactly the same as in \cite{HS:HyperellipticHasseWitt} and \cite{Harvey:HyperellipticPolytime}; we omit the details.
\end{proof}

In \cite{HS:HyperellipticHasseWitt} we also gave an algorithm \textsc{RemainderForest}, which has the same input and output specifications as \textsc{RemainderTree}, but saves space by splitting the work into $2^\kappa$ subtrees, where $\kappa \in [0, \ell]$ is a parameter.
(In \cite{HS:HyperellipticHasseWitt} the parameter $\kappa$ was called $k$.)
This is crucial for practical computations, as \textsc{RemainderTree} is extremely memory-intensive.

Theorem \ref{thm:RT} leads to the following complexity bound for \textsc{RemainderForest}, improving Theorem 4.2 of \cite{HS:HyperellipticHasseWitt}.
We omit the proof, which is exactly the same as in \cite{HS:HyperellipticHasseWitt}, with obvious modifications to take account of Theorem \ref{thm:RT}.
\begin{theorem}
\label{thm:RF}
Let $B$ be an upper bound on the bit-size of $\,\prod_{j=0}^{b-1}m_j$ such that $B/2^\kappa$ is an upper bound on the bit-size of $\prod_{j=st}^{st+t-1}m_j$ for all $s$, where $t := 2^{\ell - \kappa}$.
Let $B'$ be an upper bound on the bit-size of any entry of $V$,
and let~$H$ be an upper bound on the bit-size of any $m_0,\ldots,m_{b-1}$ and any entry in $A_0,\ldots,A_{b-1}$.
Assume that $\log r = O(H)$ and that $r = O(\log b)$.
The running time of the \textsc{RemainderForest} algorithm is
\[
O(r^2\M(B+bH)(\ell-\kappa) + 2^\kappa r^2\M(B)+r\M(B')),
\]
and its space complexity is
\[
O(2^{-\kappa}r^2(B+bH)(\ell-\kappa)+r(B+B')).
\]
\end{theorem}

\section{Computing the first row}
\label{sec:algs}

As above we work with a hyperelliptic curve $C/\Q$ of genus $g$ defined by $y^2=f(x)$, where $f(x)=\sum_{i=c}^d f_ix^i \in \Z[x]$ is squarefree and $f_cf_d\ne 0$.
We define $h(x):=f(x)/x^c=\sum_{i=0}^{r} h_ix_i$ with $r:=d-c$ and put $e:=2-c\in\{1,2\}$, as in \S\ref{sec:RR}.
We call a prime~$p$ \emph{admissible} if it is a prime of good reduction that does not divide~$h_0$.
The following algorithm uses \eqref{eq:main-primed} to compute $W_p^1$, the first row of the Hasse--Witt matrix $W_p$, simultaneously for admissible primes $p\le N$.

\medskip

\noindent
\textbf{Algorithm} \textsc{ComputeHasseWittFirstRows}
\vspace{2pt}

\noindent
Given an integer $N>1$ and a hyperelliptic curve $y^2=f(x)$ with $f(x), h(x)$, $r$, and~$e$ as above, compute $W_p^1$ for all admissible primes $p\le N$ as follows:
\smallskip

\begin{enumerate}[1.]
\setlength{\itemsep}{2pt}
\item
Let $\ell=\lceil \log_2 N\rceil - 1$ and initialize a sequence of moduli $m_n$ by
\[
m_n :=
\begin{cases}
p & \text{if $p = 2n+1$ is an admissible prime in $[1,N]$}, \\
1 & \text{otherwise},
\end{cases}
\]
for all integers $n \in [1, 2^\ell)$.

\item
Run \textsc{RemainderForest} with $b:=2^\ell$ and $\kappa := \lceil 2 \log_2 \ell \rceil$ on inputs $V:=V_0$, $A_k := M'_{k+1}$, and $m_k$ with $k\in[0,b)$, where $V_0$ and $M'_k$ are as defined in \S\ref{sec:RR}, to compute
\[
u_n := V_0 M'_1 \cdots M'_n \bmod m_n,
\]
for all integers $n\in [1,N/2)$. (One may set $A_k$ to the zero matrix for $k\ge N/2$).

\item
Similarly, use \textsc{RemainderForest} to compute $\delta_n := (en)! \bmod m_n$ for all $n \in [1, N/2)$.
\newline
(For $e=2$ one can skip this step and simply set $\delta_n:=-1$; see Remark \ref{rem:root}.)

\item
For each admissible prime $p=2n+1\in [1,N]$ output the last $g$ entries of the vector
\[
\frac{\kron{2}{p}^e}{\kron{h_0}{p}^{e-1} \delta_n} u_n \bmod m_n
\]
in reverse order (this is the vector $W_p^1$).

\end{enumerate}
\medskip

\begin{remark}
The bulk of the work in \textsc{ComputeHasseWittFirstRows} occurs in step~2.
In the $e=2$ case, this step is about twice as expensive as the $e=1$ case, because there are twice as many matrices $M_k$; equivalently, the entries of $M'_k$ are about twice as big.
This suggests that it is advantageous to change variables, if possible, to achieve $e=1$.
(Step~3 is more expensive when $e=1$, but the extra cost is negligible compared to the savings achieved in step 2.)
This can be done if the curve has a rational Weierstrass point, by moving the Weierstrass point to $x = 0$.
In \S\ref{sec:ratws} we discuss further optimizations along these lines that may utilize up to $g+1$ rational Weierstrass points.
\end{remark}

\begin{theorem}\label{thm:firstrows}
Assume that $g = O(\log N)$ and that $\log \max_i |f_i| = O(\log N)$. The algorithm \textsc{ComputeHasseWittFirstRows} runs in $O(g^2 \M(N\log N)\log N)$ time and $O(g^2 N)$ space.
\end{theorem}
\begin{proof}
Let us analyze the complexity of step 2, using Theorem~\ref{thm:RF}.
The prime number theorem implies that we may take $B = O(N)$, and the hypothesis on $B/2^\kappa$ follows easily (see the proof of \cite[Theorem~1.1]{HS:HyperellipticHasseWitt}).
We also have $B' = O(1)$ and
\[
  H = O(\log N + \log g + \log \max_i |h_i|) = O(\log N),
\]
by the definition of the matrices $M'_k$.
The remaining hypotheses are satisfied since we have $r = O(g) = O(\log N)$.
By Theorem~\ref{thm:RF} and our choice of $\kappa$, step 2 runs in time
\[
  O(g^2 \M(N \log N) \log N + (\log N)^2 g^2 \M(N));
\]
since $\M(s)/(s \log s)$ is increasing, this simplifies to $O(g^2 \M(N \log N) \log N)$.
The space complexity of step 2 is
\[
  O((\log N)^{-2} g^2 (N \log N) \log N + gN) = O(g^2 N).
\]

The second invocation of \textsc{RemainderForest} (step 3) is certainly no more expensive than the first. The remaining steps such as enumerating primes and computing quadratic residue symbols take negligible time (see the proof of \cite[Theorem~1.1]{HS:HyperellipticHasseWitt}).
\end{proof}

\begin{remark}
In practice, the parameter $\kappa$ is chosen based on empirical performance considerations, rather than strictly according to the formula given above.
\end{remark}

\begin{remark}
The $O(g^2 N)$ space complexity of \textsc{ComputeHasseWittFirstRows} is larger than the $O(gN)$ size of its output.
Most of this space can be reused in subsequent calls to \textsc{ComputeHasseWittFirstRows}, so we only need $O(g^2N)$ space to handle~$g$ calls, as in algorithm \textsc{ComputeHasseWittMatrices} below; this allows us to obtain the $O(g^2 N)$ space bound of Theorem~\ref{thm:main}.
\end{remark}

For inadmissible primes $p$, \textsc{ComputeHasseWittFirstRows} does not yield any information about $W_p^1$. When $f_0=0$ every good prime is admissible (if $f_0=0$ and $p$ divides $h_0=f_1$ then $f(x)$ is not squarefree modulo $p$), but when $f_0 \neq 0$ (the case $e=2-c=2$), there may be primes of good reduction that divide $f_0$ and are therefore inadmissible.

To compute $W_p^1$ at such primes, we must use an alternative algorithm.
There are many possible choices, but we present here an algorithm that uses the framework developed in \S\ref{sec:RR}.
The idea is to evaluate the product in \eqref{eq:main} in the most naive way possible.
To our knowledge, this simple algorithm has not been mentioned previously in the literature.

We first set some notation.
Let $\bar f\in \Fp[x]$ denote a polynomial for which $y^2=\bar f(x)$ defines a hyperelliptic $C_p/\Fp$ of genus $g$ (so $\bar f$ is squarefree of degree $\bar d = 2g+1$ or $\bar d =2g+2$), let $\bar{c}$ be the least integer for which $\bar f_{\bar{c}}\ne 0$, and set $\bar r :=\bar d-\bar c$.
(Note: in the case of interest, $\bar f = f\bmod p$, but $p$ divides~$f_0\ne 0$, so $\bar{c}\ne c$).
Now define $\bar h(x):= \sum_{i=0}^{\bar r} \bar h_ix^i= \bar f(x)/x^{\bar c}$ and $\bar e:= 2-\bar c$.
Let $n:= (p-1)/2$ as usual, and let $\bar M_k$ be the matrix over $\Fp$ defined in~\eqref{eq:matrix2}, with $h_i$ replaced by $\bar h_i$.
The following algorithm uses Lemma \ref{lem:Wp1} to compute the first row $W_p^1$ of the Hasse--Witt matrix of $C_p$.

\bigskip

\noindent
\textbf{Algorithm} \textsc{ComputeHasseWittFirstRow}
\vspace{2pt}

\noindent
Let $y^2=\bar f(x)$ be a hyperelliptic curve $C_p/\Fp$, with notation as above.
Compute the first row $W_p^1$ of the Hasse--Witt matrix of $C_p$ as follows:
\smallskip

\begin{enumerate}[1.]
\setlength{\itemsep}{2pt}
\item Initialize $u_0 := [0, \ldots, 0, 1] \in (\Fp)^r$ and $\delta_0 := 1 \in \Fp$.
\item For $k$ from $1$ to $\bar e n$, compute $u_k := u_{k-1} \bar M_k$ and $\delta_k := \delta_{k-1} k$.
\item Output the last $g$ entries of the vector
\[
\frac{\kron{2}{p}^{\bar e}}{\kron{\bar h_0}{p}^{\bar e-1}\delta_{\bar en}} u_{\bar en}.
\]
\end{enumerate}
\medskip

\begin{theorem}
\label{thm:firstrow}
The algorithm \textsc{ComputeHasseWittFirstRow} runs in $O(g p \M(\log p))$ time and uses $O(g \log p)$ space.
\end{theorem}
\begin{proof}
Each matrix $\bar M_k$ has at most $2r-1 = O(g)$ nonzero entries, each of which can be computed using $O(1)$ ring operations.
Each iteration of step~2 uses $O(g \M(\log p))$ bit operations, and the number of iterations is $O(p)$,
yielding a total cost of $O(g p \M(\log p))$ for step~2.
The Legendre symbol and division in step 3 require at most $O(\log^2 p)$ bit operations, which is negligible.

For the space bound, each $u_k$ and $\delta_k$ may overwrite $u_{k-1}$ and $\delta_{k-1}$, respectively.
Each~$\bar M_k$ requires just $O(g\log p)$ space and can be computed as needed and then discarded.
\end{proof}

The space bound in Theorem~\ref{thm:firstrow} is optimal; in fact, it matches the size of both the input and the output.
In practice, this algorithm performs quite well for small to moderate $p$, primarily due its extremely small memory footprint; see \S\ref{sec:perf} for performance details.

\section{Hasse--Witt matrices of translated curves}
\label{sec:Wpa}
In this section we fix a prime $p$ of good reduction for our hyperelliptic equation $y^2 = f(x)$.
For each integer $a$, let $W_p(a)=[w_{ij}(a)]$ denote the Hasse--Witt matrix of the translated curve $y^2=f(x+a)$ at $p$, and let $W_p^1(a)$ denote its first row.
In this section we show that if we know $W_p^1(a_i)$ for integers $a_1, \ldots, a_g$ that are distinct modulo $p$, then we can deduce the entire Hasse--Witt matrix $W_p=[w_{ij}]$ of our original curve at $p$.

We first study how $W_p$ transforms under the translation $x \mapsto x + a$.
\begin{theorem}
With notation as above we have
\[
W_p(a) = T(a)W_pT(-a),
\]
where $T(a):=[t_{ij}(a)]$ is the $g\times g$ upper triangular matrix with entries
\[
t_{ij}(a) := \binom{j-1}{i-1}a^{j-i}\qquad (1\le i,j\le g).
\]
\end{theorem}
\begin{proof}
Let $F$ be the function field of the curve $y^2=f(x)$ over $\Fp$ (the fraction field of $\Fp[x,y]/(y^2-f(x)))$.
The space $\Omega_F(0)$ of regular differentials on $F$ (as defined in \cite[\S 1.5]{Stichtenoth:FunctionFields}, for example) is a $g$-dimensional $\Fp$-vector space with basis $(\omega_1,\ldots,\omega_g)$, where
\[
\omega_i := \frac{x^{i-1}dx}{y}\qquad\qquad\,\,\, (1\le i\le g);
\]
see \cite[Ex.\ 4.6]{Stichtenoth:FunctionFields} and \cite[Eq.\ 3]{Yui:HasseWittMatrix}.
It follows from \cite[Prop.\ 2.2]{Yui:HasseWittMatrix} that the Hasse--Witt matrix $W_p(a)$ has the form $SW_pS^{-1}$, where $S=[s_{ij}]$ is the change of basis matrix from the basis~$(\omega_1,\ldots,\omega_g)$ to the basis $(\theta_1,\ldots,\theta_g)$, with
\[
\theta_j := \frac{(x+a)^{j-1}dx}{y}\qquad\, (1\le j\le g).
\]
(Note that we have replaced the matrix $S^{(p)} = [s_{ij}^p]$ that appears in \cite[Prop.\ 2.2]{Yui:HasseWittMatrix} with $S=[s_{ij}]$ because we are working over $\Fp$ and therefore have $s_{ij}^p=s_{ij}$.) 
We then have
\[
\theta_j=\frac{(x+a)^{j-1}dx}{y}=\sum_{i=1}^j\binom{j-1}{i-1}a^{j-i}\frac{x^{i-1}dx}{y}=\sum_{i=1}^g\binom{j-1}{i-1}a^{j-i}\omega_i = \sum_{i=1}^g t_{ij}(a)\omega_i,
\]
so $S=T(a)$, and $S^{-1}=T(a)^{-1}=T(-a)$.  Thus $W_p(a)=SW_pS^{-1}=T(a)W_pT(-a)$.
\end{proof}


Now suppose that $p\ge g$ and that we have computed $W_p^1(a_i)$ for integers~$a_1, \ldots, a_g$ that are distinct modulo $p$.
Writing out the equation for the first row of $W_p(a_i) = T(a_i)W_pT(-a_i)$ explicitly, we have
\[
  w_{1j}(a_i) = \sum_{k=1}^g \sum_{\ell=1}^g t_{1k}(a_i) w_{k\ell} t_{\ell j}(-a_i) = \sum_{k=1}^g \sum_{\ell=1}^j \binom{j-1}{\ell-1} (-1)^{j-\ell} a_i^{k-1+j-\ell} w_{k\ell}.
\]
As $i$ and $j$ range over $\{1, \ldots, g\}$, this may be regarded as a system of $g^2$ linear equations in the $g^2$ unknowns $w_{k\ell}$ over $\Fp$.

We claim that this system has a unique solution.
Indeed, separating out the terms with $\ell=j$, we may write
\[
 w_{1j}(a_i) = \bigl(w_{1j} + a_i w_{2j} + a_i^2 w_{3j} + \cdots + a_i^{g-1} w_{gj}\bigr) + w_j(a_i),
\]
where the last term
\begin{equation}\label{eq:wj}
w_j(a) := \sum_{k=1}^g\sum_{\ell=1}^{j-1}\binom{j-1}{\ell-1}(-1)^{j-\ell}a^{k-1+j-\ell}w_{k\ell}
\end{equation}
depends only on $a$ and the first $j-1$ columns of $W_p$.
Thus for each $j$ we have a system
\begin{equation}\label{eq:system}
\begin{bmatrix}
1 & a_1 & a_1^2 &\cdots & a_1^{g-1}\\
1 & a_2 & a_2^2 &\cdots & a_2^{g-1}\\
\vdots & \vdots &\vdots &\vdots\\
1 & a_g & a_g^2 &\cdots & a_g^{g-1}\\
\end{bmatrix}
\begin{bmatrix}
w_{1j}\\
w_{2j}\\
\vdots\\
w_{gj}
\end{bmatrix}
=
\begin{bmatrix}
w_{1j}(a_1) - w_j(a_1)\\
w_{1j}(a_2) - w_j(a_2)\\
\vdots\\
w_{1j}(a_g) - w_j(a_g)
\end{bmatrix}.
\end{equation}
The matrix on the left is a Vandermonde matrix $V(a_1,\ldots,a_g)$; it is non-singular because the $a_i$ are distinct in $\Fp$.
Therefore for each $j = 1, \ldots, g$ the system \eqref{eq:system} determines the $j$th column of $W_p$ uniquely in terms of the $w_{1j}(a_i)$ and the first $j-1$ columns of $W_p$.
Given as input $w_{1j}(a_i)$ for all $i$ and $j$, we may solve this system successively for $j = 1, \ldots, g$ to determine all $g$ columns of $W_p$.

\begin{example}
Consider the hyperelliptic curve
\[
y^2 = f(x) = 2x^8+3x^7+5x^6+7x^5+11x^4+13x^3+17x^2+19x+23,
\]
over the finite field $\F_{97}$, and let $a_1=0,a_2=1,a_3=2$.
Computing the first rows of the Hasse--Witt matrices $W_p(0),W_p(1),W_p(2)$ yields
\begin{eqnarray*}
w_{11}(0) = \phantom{0}9,\quad &w_{12}(0)=37,\quad &w_{13}(0)=54,\\
w_{11}(1) = 43,\quad &w_{12}(1)=60,\quad &w_{13}(1)=30,\\
w_{11}(2) = \phantom{0}5,\quad &w_{12}(2)=70,\quad &w_{13}(2)=84.
\end{eqnarray*}
Solving the system
\[
\begin{bmatrix}
1 & 0 & 0\\
1 & 1 & 1\\
1 & 2 & 4\\
\end{bmatrix}
\begin{bmatrix}
w_{11}\\
w_{21}\\
w_{31}
\end{bmatrix}
=
\begin{bmatrix}
w_{11}(0)\\
w_{11}(1)\\
w_{11}(2)\\
\end{bmatrix}
=
\begin{bmatrix}
9\\
43\\
5
\end{bmatrix}
\]
gives $w_{11} = 9, w_{21}=70, w_{31}=61$, the first column of $W_p$.
Using \eqref{eq:wj} to compute $w_2(0)=0$, $w_2(1)=54$, $w_2(2)=87$, we then solve
\[
\begin{bmatrix}
1 & 0 & 0\\
1 & 1 & 1\\
1 & 2 & 4\\
\end{bmatrix}
\begin{bmatrix}
w_{12}\\
w_{22}\\
w_{32}
\end{bmatrix}
=
\begin{bmatrix}
w_{12}(0)-w_2(0)\\
w_{12}(1)-w_2(1)\\
w_{12}(2)-w_2(2)\\
\end{bmatrix}
=
\begin{bmatrix}
37\\
6\\
80
\end{bmatrix}
\]
to get the second column $w_{12}=37$, $w_{22}=62$, $w_{32}=4$.
Finally, using \eqref{eq:wj} to compute $w_3(0)=0$, $w_3(1)=31$, $w_3(2)=88$ we solve
\[
\begin{bmatrix}
1 & 0 & 0\\
1 & 1 & 1\\
1 & 2 & 4\\
\end{bmatrix}
\begin{bmatrix}
w_{13}\\
w_{23}\\
w_{33}
\end{bmatrix}
=
\begin{bmatrix}
w_{13}(0)-w_3(0)\\
w_{13}(1)-w_3(1)\\
w_{13}(2)-w_3(2)\\
\end{bmatrix}
=
\begin{bmatrix}
54\\
96\\
93
\end{bmatrix}
\]
to get the third column $w_{13}=54$, $w_{23}=16$, $w_{33}=26$, and we have
\[
W_p = \begin{bmatrix}
9&37&54\\
70&62&16\\
61&4&26
\end{bmatrix},
\]
which is the Hasse--Witt matrix of $y^2=f(x)$.
\end{example}

We now bound the complexity of this procedure.  The bound given in the lemma below is likely not the best possible, but it suffices for our purposes here.

\begin{lemma}\label{lemma:WpFromFirstRows}
Given $W_p^1(a_1),\ldots, W_p^1(a_g)$, we may compute $W_p$ in
\[
O((g^3 + g^2 \log \log p) \M(\log p))
\]
time and $O(g^2\log p)$ space.
\end{lemma}
\begin{proof}
We first invert the Vandermonde matrix $V(a_1,\ldots,a_g)$.
This requires $O(g^2)$ field operations in $\Fp$, by \cite{EF:Vandermonde}, or $O(g^2 \M(\log p) \log \log p)$ bit operations.

To compute $W_p$, we use the algorithm sketched above.
More precisely, let us define $\beta_j(a) := \sum_{k=1}^g a^{k-1} w_{kj}$, so that
\[
 w_j(a_i) = \sum_{\ell=1}^{j-1} \binom{j-1}{\ell-1}(-a_i)^{j-\ell} \beta_\ell(a_i).
\]
Having computed column $j-1$ of $W_p$, we may compute $\beta_{j-1}(a_i)$ for all $i$, at a cost of $O(g^2)$ ring operations in $\Fp$.
We then use the formula above to compute $w_j(a_i)$ for all~$i$, again at a cost of $O(g^2)$ ring operations.
Finally we solve \eqref{eq:system} for the $j$th column of $W_p$, using the known inverse of $V(a_1,\ldots,a_g)$, with $O(g^2)$ ring operations.
This procedure is repeated for $j=1, \ldots, g$, for a total of $O(g^3)$ ring operations, or $O(g^3 \M(\log p))$ bit operations.

The space bound is clear; we only need space for $O(g^2)$ elements of $\Fp$.
\end{proof}


\section{Computing the whole matrix}
\label{sec:whole}

In this section we assemble the various components we have developed to obtain an algorithm for computing the whole Hasse--Matrix matrix $W_p$.
We begin with an algorithm that handles a single prime~$p \geq g$.
\bigskip
\pagebreak[1]

\noindent
\textbf{Algorithm} \textsc{ComputeHasseWittMatrix}
\vspace{2pt}

\noindent
Given a hyperelliptic curve $y^2=f(x)$ over $\Fp$ of genus $g \le p$, compute the Hasse--Witt matrix $W_p$ as follows:
\smallskip

\begin{enumerate}[1.]
\setlength{\itemsep}{2pt}
\item For $g$ distinct $a_i\in\Fp$, compute $W_p^1(a_i)$ by applying \textsc{ComputeHasseWittFirstRow} to the equation $y^2 = f(x + a_i)$.
\item Deduce $W_p$ from $W_p^1(a_1), \ldots, W_p^1(a_g)$ using Lemma~\ref{lemma:WpFromFirstRows}.
\end{enumerate}
\smallskip

\begin{theorem}
\label{thm:oneprime}
The algorithm \textsc{ComputeHasseWittMatrix} runs in $O(g^2p \M(\log p))$ time and $O(g^2\log p)$ space.
\end{theorem}
\begin{proof}
For each $i$, computing the polynomial $f(x + a_i)$ costs $O(g^2)$ ring operations in $\Fp$, and this is dominated by the $O(g p \M(\log p))$ bit complexity of computing $W_p^1(a_i)$ (Theorem~\ref{thm:firstrow}), since $g \leq p$.
Thus we can compute $W_p^1(a_1), \ldots, W_p^1(a_g)$ in $O(g^2 p \M(\log p))$ time.
By Lemma~\ref{lemma:WpFromFirstRows} we may deduce $W_p$ in time
\[
O((g^3 + g^2 \log \log p) \M(\log p)) = O(g^2 p \M(\log p)).
\]
The space bound follows immediately from the bounds in Theorem~\ref{thm:firstrow} and Lemma~\ref{lemma:WpFromFirstRows}.
\end{proof}

Finally we arrive at \textsc{ComputeHasseWittMatrices}, which computes $W_p$ for all good primes $p\le N$ by computing $W_p^1(a_i)$ for $g$ chosen integers $a_i$ and all suitable primes $p\le N$.
We rely on \textsc{ComputeHasseWittMatrix} to fill in the values of $W_p$ at any good primes $p$ that are  inadmissible for one of the translated curve $y^2=f(x+a_i)$ or for which the $a_i$ are not distinct modulo $p$.
\bigskip

\noindent
\textbf{Algorithm} \textsc{ComputeHasseWittMatrices}
\vspace{2pt}

\noindent
Given a hyperelliptic curve $y^2=f(x)$ with $f\in\Z[x]$, and an integer~$N>1$, compute the Hasse--Witt matrices $W_p$ for all primes $p\le N$ of good reduction as follows:
\smallskip

\begin{enumerate}[1.]
\setlength{\itemsep}{2pt}
\item For odd primes $p < g$ of good reduction compute $W_p$ directly from its definition by expanding $f(x)^{(p-1)/2}\bmod p$ and selecting the appropriate coefficients.
\item Choose $g$ distinct integers $a_1,\ldots,a_g$ that are either roots of $f(x)$ or in the interval $[0,g)$.
Let $\SSS$ be the set of primes $p \in [g,N]$ of good reduction that divide some $a_i - a_j$ or some nonzero $f(a_i)$.
\item For primes $p \in \SSS$, use \textsc{ComputeHasseWittMatrix} to compute $W_p$.
\item Use \textsc{ComputeHasseWittFirstRows} to compute $W_p^1(a_1),\ldots,W_p^1(a_g)$ for all primes $p\in[g,N]$ of good reduction that do not lie in~$\SSS$.
Then deduce $W_p$ for each such prime using Lemma~\ref{lemma:WpFromFirstRows}.
\end{enumerate}

\begin{remark}\label{rem:interleave}
As in \cite{HS:HyperellipticHasseWitt}, the execution of the $g$ calls to \textsc{ComputeHasseWittFirstRows} in step 4 may be interleaved so that $W_p^1(a_1),\ldots,W_p^1(a_g)$ are computed for batches of primes~$p$ corresponding to subtrees of the remainder forest, and the computation of the matrices $W_p$ for these primes can then be completed batch by batch.
\end{remark}

We now prove the main theorem announced in \S\ref{sec:intro},
which states that \textsc{ComputeHasseWittMatrices} runs in $O(g^3\M(N\log N)\log N)$ time and $O(g^2 N)$ space, under the hypotheses $g\log g=O(\log N)$ and $\log \max_i|f_i|=O(\log N)$.

In order to simplify the analysis, \emph{we assume that we always choose $a_i = i-1$} in step 2.
(The complexity bounds of the theorem hold without this assumption, i.e., if we allow $a_i$ to be any root of $f(x)$, but we do not prove this.)

\begin{proof}[Proof of Theorem~\ref{thm:main}]
The time complexity of step 1 is $O(g \M(\log N)+\M(g^2\log g))$ for each prime;
the first term covers the cost of reducing $f(x)$ modulo $p$, and the second covers the cost of computing $f(x)^{(p-1)/2}$ in the ring $\Fp[x]$.
There are at most $g$ primes, so the overall cost is $O(g^2 \M(\log N) + g \M(g^2\log g))$.
The space complexity is $O(g^2 \log g)$ for each prime, and also $O(g^2 \log g)$ overall.
Both bounds are dominated by the bounds given in the theorem.

The coefficient of $x^j$ in the translated polynomial $f^{(i)}(x) := f(x + a_i)$ is $\sum_k \binom{k}{j} a^{k-j} f_k$, thus
\[
\max_j |(f^{(i)})_j| \leq (2g+2) \binom{2g+2}{g+1} g^{2g+2} \max_k |f_k|,
\]
and therefore $\log \max_j |(f^{(i)})_j| = O(g \log g + \log \max_k |f_k|) = O(\log N)$.
In particular, the bit-size of $f(a_i)$ is $O(\log N)$, so the number of primes that divide any particular $f(a_i)$ is $O(\log N)$.
Consequently, $|\SSS| = O(g \log N)$.

By Theorem~\ref{thm:oneprime}, the total time spent in step 3 is $O((g \log N) (g^2 N \M(\log N)))$, which is dominated by $O(g^3 \M(N \log N) \log N)$.
The space used in step 3 is negligible.

Finally, computing $W_p^1(a_1),\ldots, W_p^1(a_g)$ for suitable primes $p\le N$ in step 4 requires time $O(g^3\M(N\log N)\log N)$ and space $O(g^2N)$, by Theorem~\ref{thm:firstrow}.
This dominates the contribution from Lemma~\ref{lemma:WpFromFirstRows}, which is at most $O((N/\log N)(g^3 + g^2 \log \log N) \M(\log N))$ over all primes.
The space complexity of the latter is also negligible.
\end{proof}


\subsection{Optimizations for curves with rational Weierstrass points}\label{sec:ratws}

For hyperelliptic curves with one or more rational Weierstrass points the complexity of \textsc{ComputeHasseWittMatrices} can be improved by a significant constant factor.
For curves with a rational Weierstrass point $P$, we can ensure that $d=2g+1$ by putting $P$ at infinity.
This also ensures that every translated curve $y^2=f(x+a_i)$ also has $d=2g+1$, and we get an overall speedup by a factor of at least
\[
\left(\frac{2g+2}{2g+1}\right)^2
\]
compared to the case where $C$ has no rational Weierstrass points, since we work with vectors and matrices of dimension $2g+1$ rather than $2g+2$.
For example, the speedup is approximately $(6/5)^2 = 1.44$ for $g=2$ and $(8/7)^2 \approx 1.31$ for $g=3$.

Alternatively, putting $P$ at zero and choosing $a_1=0$ speeds up the computation of $W_p^1(a_1)$ by a factor of two (because we have half as many matrices $M_k$ to deal with), but it does not necessarily speed up the computation of $W_p^1(a_2),\ldots,W_p^1(a_g)$.  If we have just a single rational Weierstrass point we should put it at zero when $g\le 2$, but otherwise we should put it at infinity.

When $C$ has more than one rational Weierstrass point we can get a further performance improvement by putting rational Weierstrass points at both zero and infinity and choosing $a_1=0$.
If there are any other rational Weierstrass points, we should then choose $a_2,\ldots,a_g$ to be the negations of the $x$-coordinates of any other rational Weierstrass points.
Without loss of generality we may assume that these coordinates are all integral, since once we have $y^2=f(x)$ with a rational Weierstrass point at infinity, the polynomial $f(x)$ has odd degree and can be made monic (and integral) by scaling $x$ and $y$ appropriately.
These changes may impact the size of the coefficients of $f$, but such changes are typically small, and may even be beneficial (in any case, for sufficiently large $N$ the benefit outweighs the cost).

For each $a_i$ for which $y^2=f(x+a_i)$ has rational Weierstrass points at zero and infinity we get a speedup by a factor of
\begin{equation}\label{eq:speedup}
2\left(\frac{g+1}{g}\right)^2
\end{equation}
in the time to compute $W_p^1(a_i)$, relative to the case where $y^2=f(x)$ has no rational Weierstrass points; we get a factor of $2$ because the number of matrices $M_k$ is halved, and then a factor of $((2g+2)/(2g))^2$ from the reduction in dimension of matrices and vectors.
When $C$ has $g+1$ rational Weierstrass points we get a total speedup by the factor given in \eqref{eq:speedup}, relative to the case where there are no rational Weierstrass points.
The speedup observed in practice is a bit better than this, as may be seen in Tables~\ref{table:hw1vshw2g2} and~\ref{table:hw1vshw2g3}.  This can be explained by the fact that the cost of matrix multiplication is actually super-quadratic at the lower levels of the accumulating remainder trees.

\begin{remark}\label{rem:lastrow}
The same speedup can be achieved when there are just $g$ rational Weierstrass points, by also computing the \emph{last} row of $W_p$ and only using $g-1$ translated curves.
\end{remark}

\section{Performance results}\label{sec:perf}

We implemented our algorithms using the GNU C compiler (\texttt{gcc} version 4.8.2) and the GNU multiple precision arithmetic library (GMP version 6.0.0).
The timings listed in the tables that follow were all obtained on a single core of an Intel Xeon E5-2697v2 CPU running at a fixed clock rate of 2.70GHz with 256 GB of RAM.

\begin{table}[!htb]
\setlength{\tabcolsep}{6pt}
\begin{tabular}{@{}rrrrrrrrrrrr@{}}
&&\multicolumn{2}{c}{$w=0$}&&\multicolumn{2}{c}{$w=1$}&&\multicolumn{2}{c}{$w=2$}&\\
\cmidrule(r){3-5}\cmidrule(r){6-8}\cmidrule(r){9-11}
$N$&&\texttt{hw1}&\texttt{hw2}&&\texttt{hw1}&\texttt{hw2}&&\texttt{hw1}&\texttt{hw2}\\
\midrule
$2^{14}$&&0.8&0.2&&0.5&0.1&&0.3&0.1\\
$2^{15}$&&2.6&0.6&&1.2&0.3&&0.6&0.2\\
$2^{16}$&&5.8&1.6&&3.2&0.8&&1.6&0.4\\
$2^{17}$&&14.0&4.1&&8.1&2.2&&4.1&1.0\\
$2^{18}$&&33.1&9.5&&20.4&5.1&&9.7&2.3\\
$2^{19}$&&81.3&21.8&&49.6&12.1&&23.5&5.3\\
$2^{20}$&&\pnum{192}&51.3&&\pnum{116}&28.2&&56.4&12.6\\
$2^{21}$&&\pnum{470}&\pnum{122}&&\pnum{274}&66.7&&\pnum{142}&29.0\\
$2^{22}$&&\pnum{1183}&\pnum{280}&&\pnum{638}&\pnum{155}&&\pnum{335}&67.6\\
$2^{23}$&&\pnum{2830}&\pnum{654}&&\pnum{1510}&\pnum{353}&&\pnum{789}&\pnum{160}\\
$2^{24}$&&\pnum{6500}&\pnum{1520}&&\pnum{3500}&\pnum{845}&&\pnum{1820}&\pnum{347}\\
$2^{25}$&&\pnum{15000}&\pnum{3460}&&\pnum{8190}&\pnum{1890}&&\pnum{4240}&\pnum{834} \\
$2^{26}$&&\pnum{34100}&\pnum{7480}&&\pnum{18700}&\pnum{4280}&&\pnum{9620}&\pnum{1870} \\
\bottomrule
\end{tabular}
\bigskip

\caption{Comparison of old (\texttt{hw1}) and new (\texttt{hw2}) average polynomial-time algorithms for genus 2 curves over $\Q$ with $w$ rational Weierstrass points (times in CPU seconds).}\label{table:hw1vshw2g2}
\end{table}

\begin{table}[!htb]
\setlength{\tabcolsep}{4pt}
\begin{tabular}{@{}rrrrrrrrrrrrrr@{}}
&&\multicolumn{2}{c}{$w=0$}&&\multicolumn{2}{c}{$w=1$}&&\multicolumn{2}{c}{$w=2$}&&\multicolumn{2}{c}{$w=3$}&\\
\cmidrule(r){3-5}\cmidrule(r){6-8}\cmidrule(r){9-11}\cmidrule(r){12-14}
$N$&&\texttt{hw1}&\texttt{hw2}&&\texttt{hw1}&\texttt{hw2}&&\texttt{hw1}&\texttt{hw2}&&\texttt{hw1}&\texttt{hw2}\\
\midrule
$2^{14}$&&3.3&0.5&&2.3&0.4&&1.5&0.3&&1.4&0.2\\
$2^{15}$&&10.8&1.5&&6.1&1.0&&5.1&0.7&&3.7&0.5\\
$2^{16}$&&25.9&4.6&&16.8&2.9&&10.0&2.1&&9.9&1.2\\
$2^{16}$&&62.1&12.6&&40.4&7.8&&23.2&5.5&&23.6&3.3\\
$2^{18}$&&\pnum{147}&28.9&&96.1&17.3&&57.1&12.6&&56.7&7.7\\
$2^{19}$&&\pnum{347}&68.1&&\pnum{230}&42.7&&\pnum{141}&30.2&&\pnum{139}&18.5\\
$2^{20}$&&\pnum{878}&\pnum{156}&&\pnum{544}&99.4&&\pnum{326}&68.2&&\pnum{329}&42.6\\
$2^{21}$&&\pnum{1950}&\pnum{363}&&\pnum{1280}&\pnum{231}&&\pnum{792}&\pnum{161}&&\pnum{782}&97.1\\
$2^{22}$&&\pnum{4500}&\pnum{841}&&\pnum{3130}&\pnum{528}&&\pnum{1840}&\pnum{370}&&\pnum{1820}&\pnum{225}\\
$2^{23}$&&\pnum{10700}&\pnum{1920}&&\pnum{7370}&\pnum{1260}&&\pnum{4380}&\pnum{859}&&\pnum{4330}&\pnum{533}\\
$2^{24}$&&\pnum{24300}&\pnum{4360}&&\pnum{16800}&\pnum{2830}&&\pnum{10200}&\pnum{2010}&&\pnum{9960}&\pnum{1200}\\
$2^{25}$&&\pnum{60400}&\pnum{9910}&&\pnum{39000}&\pnum{6220}&&\pnum{23800}&\pnum{4430}&&\pnum{2320}&\pnum{2710}\\
$2^{26}$&&\pnum{128000}&\pnum{21000}&&\pnum{83900}&\pnum{13700}&&\pnum{53400}&\pnum{9930}&&\pnum{53100}&\pnum{5980}\\
\bottomrule
\end{tabular}
\bigskip

\caption{Comparison of old (\texttt{hw1}) and new (\texttt{hw2}) average polynomial-time algorithms for genus 3 hyperelliptic curves over $\Q$ with $w$ rational Weierstrass points (times in CPU seconds).}\label{table:hw1vshw2g3}
\end{table}

\begin{table}[!htb]
\setlength{\tabcolsep}{8pt}
\begin{tabular}{@{}rrrrrrrr@{}}
&&\multicolumn{2}{c}{genus 2}&&\multicolumn{2}{c}{genus 3}&\\
\cmidrule(r){3-4}\cmidrule(r){6-8}
$N$&&\textbf{\texttt{sj}}&\textbf{\texttt{hw2}}&&\textbf{\texttt{hf}}&\textbf{\texttt{hw2}}\\
\midrule
$2^{14}$&&0.2&0.1&&7.2&0.4\\
$2^{15}$&&0.6&0.3&&16.3&1.0\\
$2^{16}$&&1.7&0.9&&39.1&2.9\\
$2^{17}$&&5.5&2.2&&98.3&7.8\\
$2^{18}$&&19.2&5.3&&\pnum{255}&18.3\\
$2^{19}$&&78.4&12.5&&\pnum{695}&43.2\\
$2^{20}$&&\pnum{271}&27.8&&\pnum{1950}&98.8\\
$2^{21}$&&\pnum{1120}&64.5&&\pnum{5600}&\pnum{229}\\
$2^{22}$&&\pnum{2820}&\pnum{155}&&\pnum{16700}&\pnum{537}\\
$2^{23}$&&\pnum{9840}&\pnum{357}&&\pnum{51200}&\pnum{1240}\\
$2^{24}$&&\pnum{31900}&\pnum{823}&&\pnum{158000}&\pnum{2800}\\
$2^{25}$&&\pnum{105000}&\pnum{1890}&&\pnum{501000}&\pnum{6280}\\
$2^{26}$&&\pnum{349000}&\pnum{4250}&&\pnum{1480000}&\pnum{13900}\\
$2^{27}$&&\pnum{1210000}&\pnum{9590}&&\pnum{4360000}&\pnum{31100}\\
$2^{28}$&&\pnum{4010000}&\pnum{21200}&&\pnum{12500000}&\pnum{69700}\\
$2^{29}$&&\pnum{13200000}&\pnum{48300}&&\pnum{39500000}&\pnum{155000}\\
$2^{30}$&&\pnum{45500000}&\pnum{108000}&&\pnum{120000000}&\pnum{344000}\\
\bottomrule
\end{tabular}
\bigskip

\caption{Comparison of the new average polynomial-time algorithm (\texttt{hw2}) to \texttt{smalljac} (\texttt{sj}) in genus 2 and \texttt{hypellfrob} (\texttt{hf}) in genus 3 for hyperelliptic curves over $\Q$ with one rational Weierstrass point (times in CPU seconds).
For $N>2^{26}$ the \texttt{sj} and \texttt{hf} timings were estimated by sampling $p\le N$.}\label{table:sjhfcmp}
\end{table}

\begin{table}[!htb]
\setlength{\tabcolsep}{4pt}
\begin{tabular}{@{}rrrrrrrrrrrrrrrrr@{}}
&&\multicolumn{2}{c}{$N=2^{16}$}&&\multicolumn{2}{c}{$N=2^{18}$}&&\multicolumn{2}{c}{$N=2^{20}$}&&\multicolumn{2}{c}{$N=2^{22}$}&&\multicolumn{2}{c}{$N=2^{24}$}&\\
\cmidrule(r){3-5}\cmidrule(r){6-8}\cmidrule(r){9-11}\cmidrule(r){12-14}\cmidrule(r){15-17}
$g$&&\textbf{\texttt{hf}}&\textbf{\texttt{hw2}}&&\textbf{\texttt{hf}}&\textbf{\texttt{hw2}}&&\textbf{\texttt{hf}}&\textbf{\texttt{hw2}}&&\textbf{\texttt{hf}}&\textbf{\texttt{hw2}}&&\textbf{\texttt{hf}}&\textbf{\texttt{hw2}}\\
\midrule
3&&39&3&&255&18&&\num{1950}&99&&\num{16700}&537&&\num{158000}&\num{2800}\\
4&&77&9&&479&60&&\num{3550}&322&&\num{30000}&\num{1680}&&\num{277000}&\num{8640}\\
5&&140&18&&836&136&&\num{5990}&694&&\num{48900}&\num{3590}&&\num{440000}&\num{18000}\\
6&&239&28&&\num{1360}&278&&\num{9330}&\num{1400}&&\num{74200}&\num{7340}&&\num{661000}&\num{35200}\\
7&&375&44&&\num{2070}&492&&\num{13800}&\num{2460}&&\num{106000}&\num{12500}&&\num{949000}&\num{60600}\\
8&&570&63&&\num{3060}&825&&\num{19800}&\num{4310}&&\num{147000}&\num{21400}&&\num{1330000}&\num{103000}\\
9&&835&89&&\num{4410}&\num{1400}&&\num{27500}&\num{7120}&&\num{200000}&\num{34200}&&\num{1780000}&\num{166000}\\
10&&\num{1189}&122&&\num{6060}&\num{2230}&&\num{37400}&\num{10900}&&\num{273000}&\num{53700}&&\num{2340000}&\num{259000}\\
\bottomrule
\end{tabular}
\bigskip

\caption{Comparison of the new average polynomial-time algorithm (\texttt{hw2}) to \texttt{hypellfrob} (\texttt{hf}) for hyperelliptic curves over $\Q$ of genus 3 to 10 with one rational Weierstrass point (times in CPU seconds).}\label{table:highergenus}
\end{table}

In our tests we used curves with small coefficients; we generally set $f_{d-i}=p_{i+1}$, where~$p_i$ is the $i$th prime, implying that $\log\max|f_i|=O(g\log g)$.
For curves with $w>1$ rational Weierstrass points we chose $f(x)$ monic with integer roots at $0,\ldots, w-2$.
For genus 3 curves with 3 rational Weierstrass points we applied Remark~\ref{rem:lastrow}.

Tables \ref{table:hw1vshw2g2}-\ref{table:highergenus} compare the performance of the new average polynomial-time algorithm \textsc{ComputeHasseWittMatrices} to the average polynomial-time algorithm of \cite{HS:HyperellipticHasseWitt}, and also to the \texttt{smalljac} library \cite{smalljac} based on \cite{KS:Lseries}, which was previously the fastest available package for performing these computations in genus $g\le 2$ (within the feasible range of $N$), and the \texttt{hypellfrob} library \cite{hypellfrob} based on \cite{Har-kedlaya}, which was previously the fastest available package for performing these computations in genus $g \ge 3$.

\begin{table}[!htb]
\setlength{\tabcolsep}{6pt}
\begin{tabular}{@{}rrrrrrrrrrrrrrrrr@{}}
&&\multicolumn{2}{r}{$p=2^{16}+1$}&&\multicolumn{2}{c}{$p=2^{17}+29$}&&\multicolumn{2}{c}{$p=2^{18}+3$}&&\multicolumn{2}{c}{$p=2^{19}+21$}&&\multicolumn{2}{c}{$p=2^{20}+7$}&\\
\cmidrule(r){3-5}\cmidrule(r){6-8}\cmidrule(r){9-11}\cmidrule(r){12-14}\cmidrule(r){15-17}
$g$&&\textbf{\texttt{hf}}&\textbf{\texttt{hwp}}&&\textbf{\texttt{hf}}&\textbf{\texttt{hwp}}&&\textbf{\texttt{hf}}&\textbf{\texttt{hwp}}&&\textbf{\texttt{hf}}&\textbf{\texttt{hwp}}&&\textbf{\texttt{hf}}&\textbf{\texttt{hwp}}\\
\midrule
3&&8&4&&11&8&&16&16&&23&33&&36&66\\
4&&16&8&&20&15&&29&31&&40&61&&61&123\\
5&&27&11&&34&22&&47&43&&65&86&&98&172\\
6&&44&19&&54&39&&74&78&&100&155&&149&310\\
7&&68&26&&82&52&&110&104&&144&207&&213&414\\
8&&102&33&&120&67&&159&134&&203&267&&295&534\\
9&&148&42&&171&84&&222&168&&279&335&&401&670\\
10&&207&42&&239&102&&303&205&&377&409&&539&819\\
11&&285&62&&325&123&&407&246&&497&492&&721&983\\
12&&381&73&&430&146&&533&292&&642&582&&965&1160\\
13&&494&86&&552&171&&685&341&&828&681&&1220&1370\\
14&&633&99&&714&197&&863&393&&1070&786&&1530&1570\\
15&&803&113&&884&225&&1070&450&&1330&899&&1910&1800\\
16&&1180&128&&1120&256&&1340&511&&1650&1020&&2330&2040\\
17&&1260&145&&1370&289&&1660&575&&1990&1150&&2820&2300\\
18&&1530&162&&1690&322&&2010&643&&2600&1280&&3330&2570\\
19&&1880&180&&2050&359&&2410&715&&2880&1430&&3930&2860\\
20&&2270&200&&2480&397&&2870&791&&3400&1580&&4630&3160\\
\bottomrule
\end{tabular}
\bigskip

\caption{Comparison of new algorithm to compute a single Hasse--Witt matrix (\texttt{hwp}) to \texttt{hypellfrob} (\texttt{hf}) for hyperelliptic curves over $\Fp$ of genus 3 to 20 with a rational Weierstrass point (times in CPU milliseconds).}\label{table:modp}
\end{table}

Table~\ref{table:modp} compares the performance of the $O(g^2p(\log p)^{1+o(1)})$ algorithm \textsc{ComputeHasseWittMatrix} to the $O(g^3p^{1/2}(\log p)^{2+o(1)})$ algorithm implemented by \texttt{hypellfrob} for computing a single Hasse--Witt matrix $W_p$ for a hyperelliptic curve of genus $g$.

While the performance data listed here focuses on running times, we should note that the new algorithm is also more space efficient than the average polynomial-time algorithm given in \cite{HS:HyperellipticHasseWitt}.
The improvement in space is not as dramatic as the improvement in time, but we typically gain a a small constant factor.
For example, the most memory intensive computation in Table~\ref{table:hw1vshw2g3} (genus 3 curves) occurs when $N=2^{26}$ and $w=0$ (no rational Weierstrass points); in this case the new algorithm (\texttt{hw2}) uses 11.4\ GB of memory, versus 22.4\ GB for the old algorithm (\texttt{hw1}).
In both cases the memory footprint can be reduced by increasing the number of subtrees used in the \textsc{RemainderForest} algorithm, as determined by the parameter $\kappa$ that appears in \S4.2 (the parameter $k$ in \cite[Table~3]{HS:HyperellipticHasseWitt}). Here we chose parameters that optimize the running time.

\section{Computing Sato--Tate distributions}

A notable application of our algorithm is the computation of Sato--Tate statistics.
Associated to each smooth projective curve $C/\Q$ of genus $g$ is the sequence of integer polynomials $L_p(T)$ at primes $p$ of good reduction that appear in the numerator of the zeta function in~\eqref{eq:zeta}.
It follows from the Weil conjectures that each normalized $L$-polynomial
\[
\overline{L}_p(T) = L_p(T/\sqrt{p})=\sum_{i=0}^{2g} a_iT^i
\]
is a real monic polynomial of degree $2g$ whose roots lie on the unit circle, with coefficients $a_i=a_{2g-i}$ that satisfy $|a_i|\le \binom{2g}{i}$.
We may then consider the distribution of the $a_i$ (jointly or individually) as $p$ varies over primes of good reduction up to a bound $N$, as $N\to\infty$.

In order to compute these Sato--Tate statistics we need to know the integer values of the coefficients of $L_p(T)$, not just their reductions modulo $p$.
As explained in~\cite{Harvey:ArithmeticSchemesPolytime,Harvey:HyperellipticPolytime}, the integer polynomial $L_p(T)$ can be computed in average polynomial time using a generalization of the method presented here.
However, for $g\le 3$ this can be more efficiently accomplished (for the feasible range of $N$) using group computations in the Jacobian of $C_p$ and its quadratic twist, as explained in \cite{KS:Lseries}.
For $g\le 2$ there are at most $5$ possible values for $L_p\in\Z[T]$ given its reduction modulo $p>13$, and the correct value can be determined in $O((\log p)^{2+o(1)})$ time, which is negligible.
For $g \le 3$ there are $O(p^{1/2})$ possible values, and the correct value can be determined in $O(p^{1/4}\M(\log p))$ time using a baby-steps giant-steps approach.
This time complexity is exponential in $\log p$ and asymptotically dominates the $O((\log p)^{4+o(1)})$ average time to compute $L_p(T)\bmod p$, but within the practical range of $N$ this is not a problem.
For example, when $N=2^{30}$ it takes approximately 344,000 CPU seconds to compute $L_p(T)\bmod p$ for all good $p\le N$ for a hyperelliptic curve of genus 3, while the time to lift $L_p(T)\bmod p$ to $\Z$ for all good $p\le N$ using the algorithm of \cite{KS:Lseries} is just 55,370 CPU seconds, far less than it would take to compute $L_p(T)$ via \cite{Harvey:ArithmeticSchemesPolytime,Harvey:HyperellipticPolytime}.

Figure \ref{fig:generic_a1} shows the distributions of the normalized $L$-polynomial coefficients $a_1,a_2$, and~$a_3$ over good primes $p\le 2^{30}$ for the curve
\[
y^2=x^7-x+1.
\]
It follows from a result of Zarhin \cite{Zarhin:LargeGaloisImage} that hyperelliptic curves of the form $y^2=x^{2g+1}-x+1$ over $\Q$ have large Galois image.
As a consequence, the Sato--Tate group of this curve, as defined in \cite{FKRS:SatoTate} or \cite{Serre:LecturesNxp}, is the unitary symplectic group $\USp(6)$.  Under the generalized Sato--Tate conjecture the distribution of normalized $L$-polynomials should match the distribution of characteristic polynomials of a random matrix in $\USp(6)$, under the Haar measure, and this indeed appears to be the case.

\begin{figure}[!htb]
\begin{center}
\includegraphics[scale=0.185]{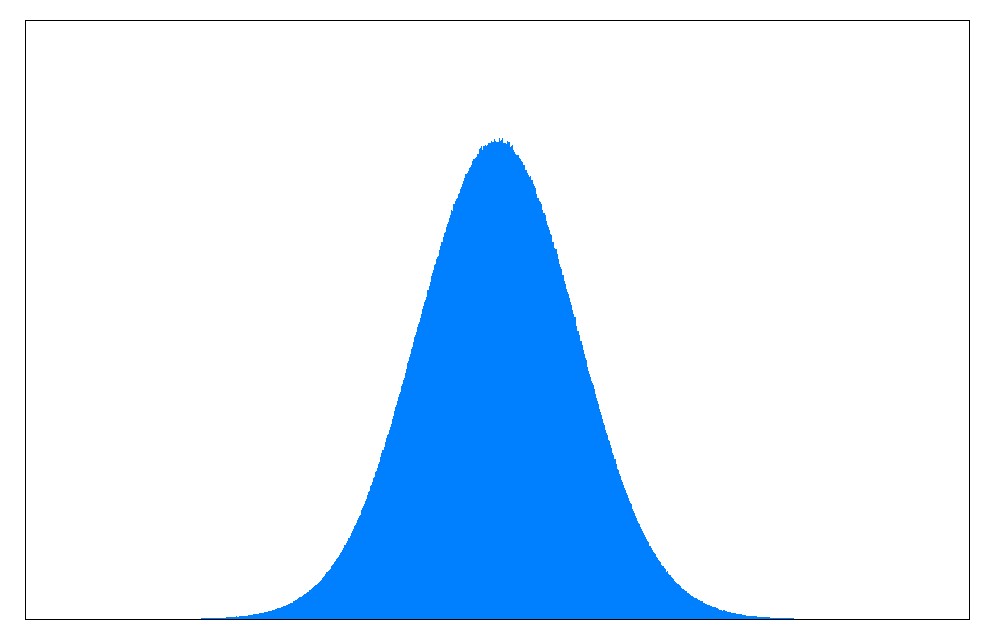}
\includegraphics[scale=0.185]{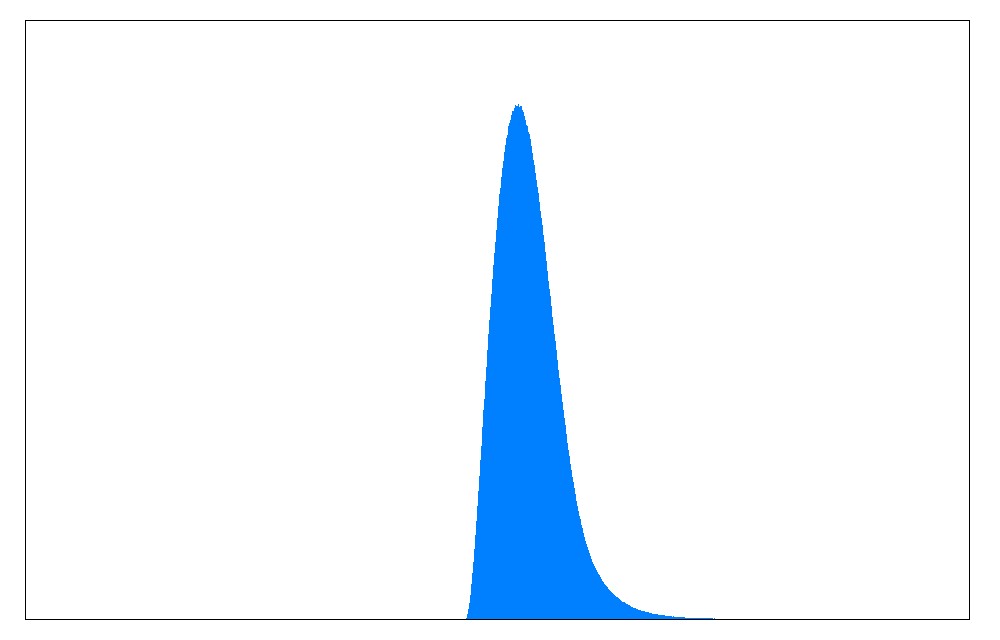}
\includegraphics[scale=0.185]{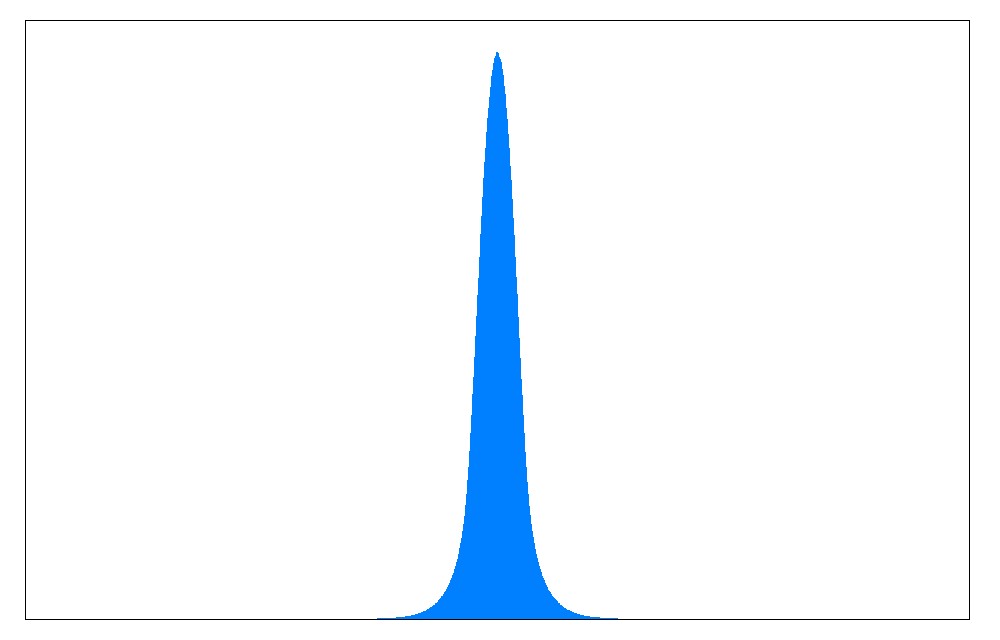}
\caption{Distributions of normalized $L$-polynomial coefficients $a_1,a_2,a_3$ for $y^2=x^7-x+1$ over primes $p\le 2^{30}$.}\label{fig:generic_a1}
\end{center}
\end{figure}

We also used our algorithm to compute Sato--Tate statistics for several other hyperelliptic curves of genus 3, including the curve
\[
y^2=x^7+3x^6+2x^5+6x^4+4x^3+12x^2+8x,
\]
which has an unusual Sato--Tate distribution as can be seen in Figure~\ref{fig:special_a1}.
This curve was found in a large search of genus 3 hyperelliptic curves with small coefficients.
This curve has a non-hyperelliptic involution $[x:y:z]\mapsto[z:y/4:x/2]$, which implies that its Jacobian has extra endomorphisms and its Sato--Tate group must be a proper subgroup of $\USp(6)$ (the exact group has yet to be determined).

\begin{figure}[!htb]
\begin{center}
\includegraphics[scale=0.185]{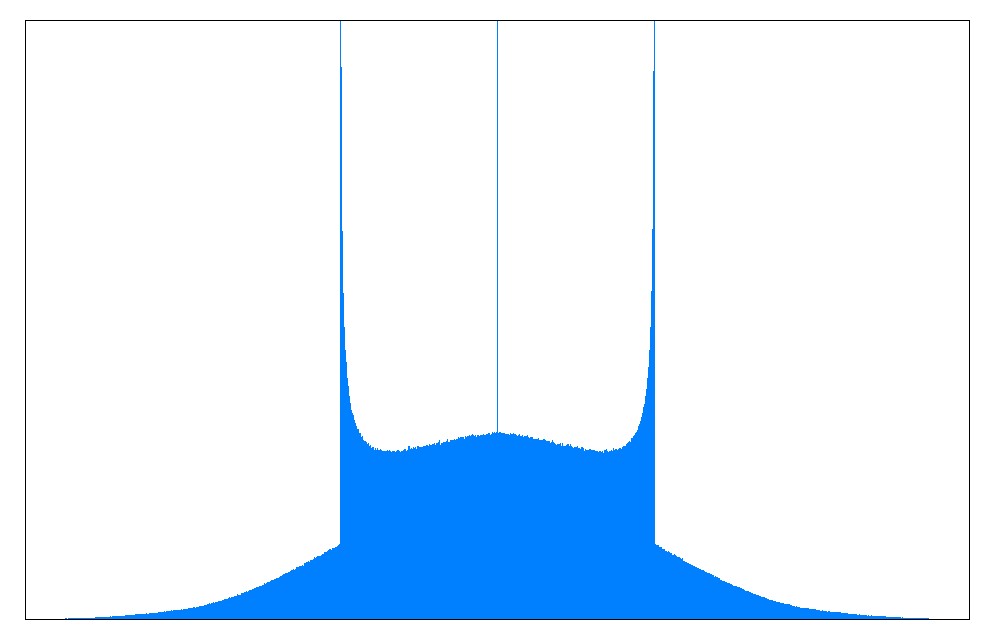}
\includegraphics[scale=0.185]{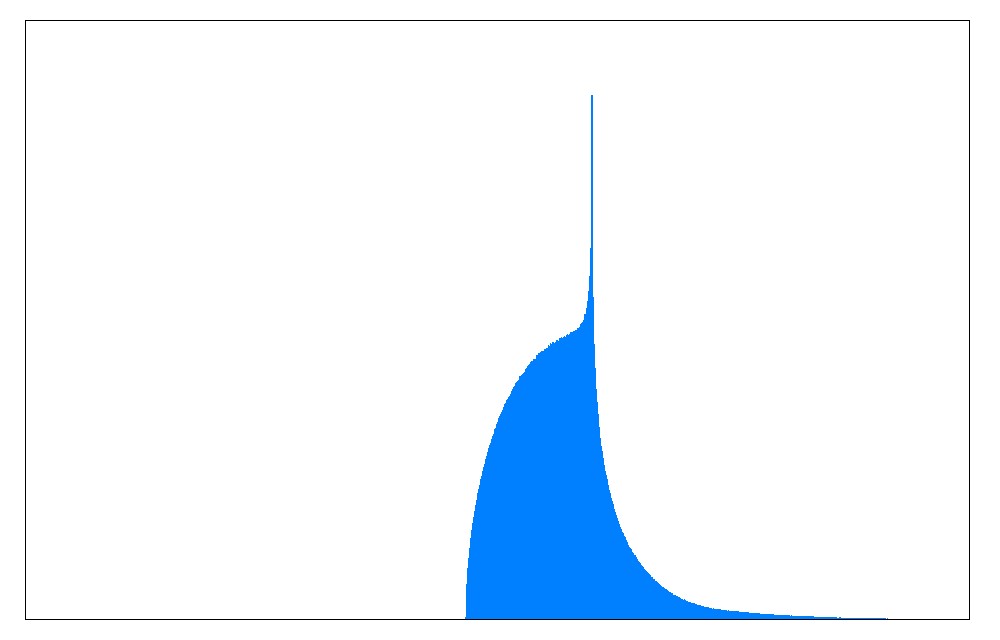}
\includegraphics[scale=0.185]{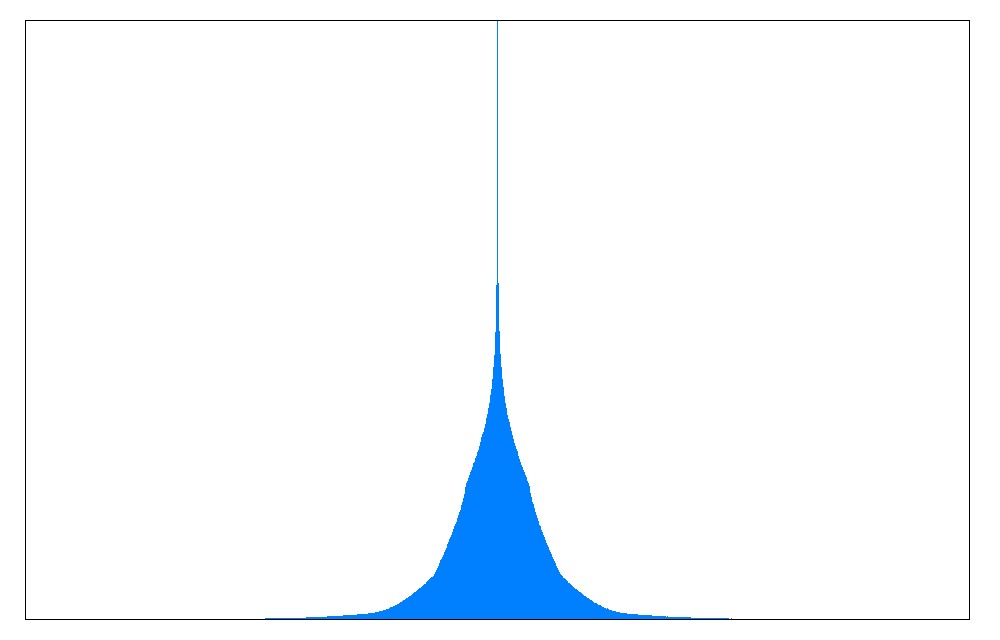}
\caption{Distributions of normalized $L$-polynomial coefficients $a_1, a_2, a_3$ for $y^2=x^7+3x^6+2x^5+6x^4+4x^3+12x^2+8x$ over good primes $p\le 2^{30}$.}\label{fig:special_a1}
\end{center}
\end{figure}

More examples can be found at \url{http://math.mit.edu/~drew}.

\bibliographystyle{amsplain}
\providecommand{\bysame}{\leavevmode\hbox to3em{\hrulefill}\thinspace}
\providecommand{\MR}{\relax\ifhmode\unskip\space\fi MR }
\providecommand{\MRhref}[2]{%
  \href{http://www.ams.org/mathscinet-getitem?mr=#1}{#2}
}
\providecommand{\href}[2]{#2}

\end{document}